\documentclass[12pt,final,reqno]{amsart}
\usepackage{ntung}

\usepackage{ntung_commands}

\newcommand{\T}{\mathbb{T}}

\newcommand{\ep}[1]{\mathrm{e}_N\!\left(#1\right)}
\newcommand{\normT}[1]{\left\lVert #1 \right\rVert_{\T}}

\begin{document}

\title{Linear equations and chromatic thresholds in $B_h$ sets}
\author{Nathan Tung}
\address{Stanford University, CA 94305, USA}
\email{ntung@stanford.edu}

\begin{abstract}
We derive sparse analogs of several Roth-type results, showing that they hold in $B_h$ sets of near-maximum size. It is shown that if a $B_h$ set is free of pairwise distinct solutions to a linear equation with more than $2h$ variables then it must be a constant factor smaller than the best-known upper bound on the size of any $B_h$ set. As a key input, it is established that extremal $B_h$ sets are Fourier pseudorandom. If the forbidden equation has a certain subdivision structure, an asymptotic saving is obtained. The case of Sidon sets ($h=2$) was previously studied by Conlon, Fox, Sudakov, and Zhao as well as Prendiville.

When forbidding a non-translation-invariant equation $E$ from a Sidon set, it is shown that if $E$ has a zero-sum subcollection of at least five coefficients then the Sidon set must either be very small or generate a Cayley graph with bounded chromatic number. On the other hand, large Sidon sets are constructed that generate Cayley graphs with unbounded chromatic number and are also free of multiple equations with zero-sum subcollections of four coefficients. This can be viewed as a sparse analog of a result of Liu, Wu, Yang, and Zhang characterizing linear equations with vanishing chromatic threshold.
\end{abstract}
\maketitle

\section{Introduction}

Call a subset $S \subset \ZZ_N$ of the cyclic group a \textit{$B_h$ set} if
    $$
    x_1 + \dots + x_h = y_1 + \dots + y_h, \quad x_1,\dots,x_h,y_1,\dots,y_h \in S
    $$
    implies equality of the multisets
    $$
    \set{x_1,\dots,x_h} = \set{y_1,\dots,y_h}.
    $$
    $B_2$ sets will also be called \textit{Sidon}. Since this definition immediately implies the $h$-fold sumset $S + \dots + S$ has size $\Omega\bigp{\abs{S}^h}$, it holds that if $S$ is $B_h$ then $\abs{S} = O\bigp{N^{1/h}}$. This paper is concerned with $B_h$ sets of size $\Theta\bigp{N^{1/h}}$. Such sets may be considered both sparse and extremal for $h \ge 2$. We are interested in proving analogs of results that are known to hold for dense subsets of $\ZZ_N$ for these sparse-extremal sets.

\subsection{Linear equations in $B_h$ sets}\label{sec:linbh}

One such result for dense sets is that dense subsets of $\ZZ_N$ must contain solutions to translation-invariant linear equations. Call a homogeneous linear equation $\sum_{i=1}^k c_i x_i = 0$ \textit{translation invariant} if $\sum_{i=1}^k c_i = 0$. Integer coefficients $c_1,\dots,c_k$ used to define linear equations will be implicitly assumed to be nonzero, so that $k$ is minimal.
\begin{theorem}[Roth \cite{roth}]\label{thm:qualvar}
        Let
        \[
        E:\ \sum_{i=1}^k c_i x_i = 0
        \]
        be such that $k \ge 3$ and $\sum_i c_i = 0$. If $S \subset \ZZ_N$ contains no solution to $E$ with pairwise distinct coordinates, then
        \[
        \abs{S} = o(N).
        \]
\end{theorem}
Obtaining more quantitative upper bounds in general and in the case of specific equations is a prominent line of work in additive combinatorics. We are not the first to seek a sparse analog of Roth's theorem; the following result concerns the case of $B_2$ sets.
\begin{theorem}[$k=5$ Conlon--Fox--Sudakov--Zhao \cite{fewfourcycle}, $k > 5$ Prendiville \cite{densesidon}]\label{thm:qualprend}
            Let
    \[
    E:\ \sum_{i=1}^k c_i x_i = 0
    \]
    be such that $k \ge 5$ and $\sum_i c_i = 0$. If $S \subset \ZZ_N$ is a Sidon set that contains no solution to $E$ with pairwise distinct coordinates, then
    \[
    \abs{S} = o\bigp{\sqrt{N}}.
    \]
\end{theorem}
Although the cited results are stated over $[N]$, the cyclic formulation follows by choosing representatives in $[N]$: any Sidon violation or solution to $E$ in the integers reduces mod $N$ to one in $\ZZ_N$. Conlon, Fox, Sudakov, and Zhao's proof is by a removal lemma for graphs with few 4-cycles. Prendiville's proof is by Fourier-analytic transference and gives a quantitative bound: a dense model is found for the rescaled indicator function $\sqrt{N}S$ and then a supersaturated version of \Cref{thm:qualvar} is used as a black box. Similar results were shown by Pascadi \cite{pascadi} and Jing, Pohoata, and Xu \cite{jing2026rothtypetheoremskstfreesets} under more general assumptions like $S$ being free of additive $K_{s,t}$ instead of Sidon (corresponding to $s=t=2$). Equations with less than five variables behave differently. Our constructions below show that the conclusion of \Cref{thm:qualprend} fails for all translation-invariant three-variable equations and for many four-variable equations.

We prove two extensions of \Cref{thm:qualprend} to $B_h$ sets for $h \ge 3$, which correspond to sets free of additive $C_{2h}$. This has been highlighted to the author as a stubborn open problem by (a subset of) the authors of \cite{fewfourcycle} and \cite{jing2026rothtypetheoremskstfreesets}, and the analogous sparse graph problem of proving a removal lemma for $C_{2h}$-free host graphs remains open when $h \ge 3$. However, our bounds either fall short of those obtained in the case of Sidon sets or require additional assumptions; an extension of analogous strength remains elusive.

Our first upper bound only requires one to forbid a sufficiently long linear equation (which may lack translation invariance) but does not obtain an asymptotic saving. Denote $C_h \coloneqq \bigp{\floor{\frac{h}{2}}!\ceil{\frac{h}{2}}!}^{1/h}$. The best-known bound for the size of the largest $B_h$ set in $\ZZ_N$ is $C_h N^{1/h} + O_h(1)$, due to Jia \cite{jia} for even $h$ and Chen \cite{chen} for odd $h$.

\begin{restatable}{theorem}{bhlin}\label{thm:solving}
Let $h\ge 3$ be an integer, and let
\[
E:\ \sum_{i=1}^k c_i x_i = 0
\]
be such that $k > 2h$. If $N$ is coprime to $c_1 \cdots c_k$ and $S \subset \ZZ_N$ is a $B_h$ set that contains no solution to $E$ with pairwise distinct coordinates, then there exists $0 < D_{h,k} < C_h$ such that
\[
\abs{S} \le \bigp{D_{h,k}+o(1)}N^{1/h}.
\]
\end{restatable}
It is only really required to assume few solutions to $E$ and not forbid them entirely. $D_{h,k}$ is made explicit in \Cref{sec:lineqbh}. It is worth noting that, for fixed $h$, as $k \to \infty$ one has $D_{h,k} \searrow 2^{-1/h}C_h$ monotonically. The proof makes use of the following result, which we believe to be of independent interest.

\begin{restatable}[Pseudorandomness]{proposition}{pseud}\label{prop:psd}
    Let $h \ge 2$ and $S \subset \ZZ_N$ be a $B_h$ set. Then for any nonzero $\xi \in \ZZ_N$,
    $$
    \abs{\widehat{S}(\xi)}^h \le C_h^h N - \abs{S}^h + O_h(\abs{S}^{h-1}).
    $$
\end{restatable}

This establishes strong Fourier pseudorandomness of extremal $B_h$ sets. Such sets then behave like the ambient group $\ZZ_N$ with regards to solving linear equations, and have roughly the expected number of solutions in a random subset of the same size. A similar pseudorandomness result was shown for Sidon subsets of $[N]$ by Ortega and Prendiville \cite{exsidon}, but their proof does not seem to easily extend to $h \ge 3$.

Our second upper bound does obtain an asymptotic saving, but only when forbidding a certain class of equations.
\begin{definition}[Subdivision]\label{def:subdivision}
    A homogeneous linear equation
    $$
    c_1 x_1 + \dots + c_k x_k = 0
    $$
    is an \textit{$h$-subdivision} if there exists a partition of its coefficients into parts of size $h$ such that all coefficients in each part are equal. That is, it can be written as
    $$
    d_1(x_{1,1}+\cdots+x_{1,h})+\cdots+d_m(x_{m,1}+\cdots+x_{m,h})=0.
    $$
\end{definition}

\begin{restatable}{theorem}{sublin}\label{thm:sublin}
Let $h\ge 3$ be an integer, and let
\[
E:\ \sum_{i=1}^k c_i x_i = 0
\]
be an $h$-subdivision with $k \ge 3h$ and $\sum_i c_i = 0$. If $S \subset \ZZ_N$ is a $B_h$ set that contains no solution to $E$ with pairwise distinct coordinates, then
\[
\abs{S} = o\bigp{N^{1/h}}.
\]
\end{restatable}

\subsection{Chromatic thresholds in Sidon sets}\label{sec:introchromthresh}

 \Cref{thm:qualprend} concerns translation-invariant equations in Sidon sets. What about non-translation-invariant equations? The dense scale result here is an additive analog of chromatic thresholds in graphs. Erd\H{o}s and Simonovits introduced in \cite{erdossimonchrom} what is now known as the \textit{chromatic threshold} of a graph $H$. Roughly speaking, it measures how much density is required in an $H$-free graph to force bounded chromatic number. The natural additive counterpart uses Cayley graphs. For $A \subset \ZZ_N$ we will denote by $(\ZZ_N,A)$ the directed \textit{Cayley graph} with an edge $u \to v$ if $v-u \in A \setminus \set{0}$, and by $\chi(\ZZ_N,A)$ the chromatic number of its underlying undirected Cayley graph. Then the chromatic threshold of a linear equation $E$ captures how large an $E$-solution-free subset $A \subseteq \FF_p$ must be to force bounded chromatic number of $(\FF_p,A)$, where $p$ is any prime. See \cite{chromthresh} for formal definitions of these thresholds and discussion about recent developments. The main result of Liu, Wu, Yang, and Zhang gives a complete classification of equations with vanishing chromatic threshold. More in line with the Roth-type results discussed in \Cref{sec:linbh}, their result can be viewed as a characterization of which equations are such that forbidding them forces a dichotomy on $A$: either it is small or structured.
\begin{theorem}[Liu--Wu--Yang--Zhang \cite{chromthresh}*{Thm.~1.4}]\label{thm:zhang}
    Let $p$ be a prime and
    \[
    E:\ \sum_{i=1}^k c_i x_i = 0
    \]
    be such that $k \ge 3$. Then the following are equivalent:
    \begin{itemize}
        \item There exists $I\subseteq [k]$ with $\abs{I}\ge 3$ and $\sum_{i\in I} c_i=0$.
        \item Any $A \subset \FF_p$ free of pairwise distinct solutions to $E$ is such that either $\abs{A} = o(p)$ or $\chi(\FF_p,A) = O(1)$.
    \end{itemize}
\end{theorem}

We use ``dichotomy'' to mean that at least one of the conclusions hold; they are not, of course, mutually exclusive.

\subsubsection{Dichotomy}

We are able to establish one direction of the natural sparse-extremal analog.

\begin{restatable}[Dichotomy]{theorem}{structchar}\label{thm:dichotomy}
    Let
    \[
    E:\ \sum_{i=1}^k c_i x_i = 0
    \]
    be such that $k\ge 5$ and there exists $I\subseteq [k]$ with $\abs{I}\ge 5$ and $\sum_{i\in I} c_i=0$. Then for every $\varepsilon>0$ there
    exists $C=C(\varepsilon,E)$ such that the following holds. For any $N$ coprime to $c_1 \cdots c_k$ and any Sidon $A\subset \ZZ_N$ with no pairwise distinct solution to $E$, either
    \[
    \abs{A} <
    \varepsilon \sqrt{N} \quad \text{or} \quad \chi\bigl(\ZZ_N,A\bigr)\le C.
    \]
\end{restatable}

This result also sheds further light on the open problem of understanding the structure of Sidon sets with near-maximum size (see, for instance, Eberhard--Manners \cite{apparentstructure} or a blog post of Gowers \cite{gowersblog}). 

\subsubsection{Constructions}\label{sec:introcons}

For the other direction, we find many obstructions showing that the assumption $\abs{I} \ge 5$ cannot in general be lowered. That is, for infinitely many $N$, we construct Sidon $A \subset \ZZ_N$ of size $\Omega\bigp{\sqrt{N}}$ such that also $\chi(\ZZ_N,A) = \omega(1)$ and $A$ is free of certain equations with zero-sum subsets of two, three, and four coefficients. Note that the Sidon equation
\begin{equation}\label{eq:sidoneq}
    x_1 + x_2 = x_3 + x_4
\end{equation}
provides an easy obstruction, as forbidding pairwise distinct solutions does not impose any extra constraints on a Sidon set; \Cref{thm:dichotomy} cannot hold for $E$ taken to be this equation, as witnessed by any large Sidon set generating a Cayley graph with unbounded chromatic number. Our contribution is to provide many obstructions beyond this trivial one.

Call a solution $x_1,\dots,x_k$ to $E$ \textit{trivial} if there exists a partition $I_1 \cup \dots \cup I_m = [k]$ such that for each $j \in [m]$, $\set{x_i}_{i \in I_j}$ are all equal and $\sum_{i \in I_j} c_i = 0$. Unlike the setting of \Cref{thm:dichotomy}, our constructions avoid all nontrivial solutions rather than merely pairwise distinct solutions. This first construction exhibits specific equations other than \eqref{eq:sidoneq} with zero-sum subsets of three and four coefficients for which \Cref{thm:dichotomy} does not hold.

\begin{restatable}{theorem}{exeq}\label{thm:exeq}
     For infinitely many odd $N$, there exists $S \subset \ZZ_N$ such that
     $$
     S \text{ is Sidon}, \quad \abs{S} \ge \bigp{\frac 1 2 - o(1)}\sqrt{N}, \quad \chi(\ZZ_N,S) = \Omega\bigp{N^{1/8}}.
     $$
     Furthermore, $S$ contains no nontrivial solution to any of
    $$
    2x_1 - x_2 - x_3 = 0, \quad 2x_1 - x_2 - x_3 + x_4 = 0, \quad 2x_1 - 2x_2 + x_3 - x_4 = 0.
    $$
\end{restatable}

Beyond these, we are able to forbid any translation-invariant equation on three variables, along with a large family of equations on four variables.

\begin{restatable}{theorem}{parabolathree}\label{thm:parabola3}
     Let $E: \sum_{i=1}^3 c_ix_i = 0$ be such that $c_1 + c_2 + c_3 = 0$. For infinitely many $N$ coprime to $c_1c_2c_3$, there exists $A \subset \ZZ_N$ such that
     $$
     A \text{ is Sidon}, \quad \abs{A} = \Omega_E\bigp{\sqrt{N}}, \quad \chi(\ZZ_N,A) = \Omega_E\bigp{N^{1/4}}.
     $$
     Furthermore, $A$ contains no nontrivial solution to $E$.
\end{restatable}

\begin{restatable}{theorem}{parabolafour}\label{thm:parabola4}
     Let $E: \sum_{i=1}^4 c_ix_i = 0$ be such that $c_1+c_2+c_3 + c_4 = 0$ and $\frac{c_1c_2}{c_3c_4}$ is not a square in $\QQ$. For infinitely many $N$ coprime to $c_1c_2c_3c_4$, there exists $A \subset \ZZ_N$ such that
     $$
     A \text{ is Sidon}, \quad \abs{A} = \Omega_E\bigp{\sqrt{N}}, \quad \chi(\ZZ_N,A) = \Omega_E\bigp{N^{1/4}}.
     $$
     Furthermore, $A$ contains no nontrivial solution to $E$.
\end{restatable}

Both chromatic number lower bounds are proved using a point-line incidence bound over $\FF_p^2$ for prime $p$, but could also be proved with spectral methods.

More generally, the dichotomy does not hold for any equation with no zero-sum subset of at least three coefficients, including those lacking translation invariance. 

\begin{restatable}{theorem}{twosum}\label{thm:twosum}
     Let $E: \sum_{i=1}^k c_i x_i = 0$ be such that there does not exist $I\subseteq [k]$ with $\abs{I}\ge 3$ and $\sum_{i\in I} c_i=0$. For infinitely many $N$ coprime to $c_1 \cdots c_k$, there exists $A \subset \ZZ_N$ such that
     $$
     A \text{ is Sidon}, \quad \abs{A} = \Omega_E\bigp{\sqrt{N}}, \quad \chi(\ZZ_N,A) = N^{\Omega_E(1)}.
     $$
     Furthermore, $A$ contains no nontrivial solution to $E$.
\end{restatable}

\subsection{Some remaining questions}

We record three natural gaps left by the present work. First, it remains unclear whether there is a translation-invariant four-variable equation for which \Cref{thm:qualprend} holds. An example not ruled out by \Cref{thm:parabola4} is
$$
2(x_1-x_2) = 3(x_3-x_4).
$$
If one desired a Sidon subset $A \subset \ZZ_N$ that was free of all nontrivial solutions to this equation, for instance, it would require that $2(A-A) \cap 3(A-A) = \set{0}$. This seems difficult to ensure if $\abs{A} = \Omega\bigp{\sqrt{N}}$. Second, for \(h\geq 3\), does forbidding an arbitrary translation-invariant equation with more than \(2h\) variables force a \(B_h\) set in \(\ZZ_N\) to have size \(o\bigp{N^{1/h}}\)? An affirmative answer would yield a common strengthening of Theorems \ref{thm:solving} and \ref{thm:sublin}. Third, it remains open whether there is some equation with no zero-sum subset of at least five coefficients for which the dichotomy of \Cref{thm:dichotomy} still holds.

\subsection{Notation and conventions}

Subsets will be identified with their indicator functions. $\widetilde{O}$ and $\widetilde{\Omega}$ will be used with arguments that grow at least polynomially to suppress factors that only grow polylogarithmically. $\ZZ_N$ will denote the cyclic group of order $N$ and $\FF_p$ the finite field of order $p$ for prime $p$. For $f: \ZZ_N \to \CC$ define $\widehat{f}: \ZZ_N \to \CC$ by
$$
\widehat{f}(\xi) \coloneqq \sum_{n \in \ZZ_N} f(n)e_N(-\xi n)
$$
where $e_N(\cdot) = e^{2\pi i \cdot/N}$. Denote
$$
\norm{f}^p_p \coloneqq \sum_{n \in \ZZ_N} \abs{f(n)}^p, \quad \norm{f}^p_{L^p} \coloneqq \EE_{n \in \ZZ_N} \abs{f(n)}^p.
$$
For $S \subset \ZZ_N$, let the representation function and energy of $S$ be
$$
r_h(x) \coloneqq \abs{\set{(s_1,\dots,s_h)\in S^h:s_1+\cdots+s_h=x}}, \quad E_h(S) \coloneqq \norm{r_h}_2^2.
$$
For $f,g: \ZZ_N \to \CC$ denote $f \ast g(x) \coloneqq \sum_{n \in \ZZ_N} f(n)g(x-n)$, so that also
$$
r_h(x) = S \ast \dots \ast S(x),
$$
where convolution is repeated $h$ times.

\section{Energy and pseudorandomness}\label{sec:ingredients}

This section collects some useful Fourier properties of $B_h$ sets and proves \Cref{prop:psd}. The lemmas will be used in full generality in \Cref{sec:lineqbh} and with $h=2$ in Sections \ref{sec:sidondichotomy} and \ref{sec:sidonconstructions}. We start with a simple consequence of Fourier inversion that is nonetheless convenient to state in a packaged form.

\begin{lemma}[\cite{chromthresh}*{Lem.~2.6}]\label{lem:inversioncount}
    Consider any equation $\sum_{i=1}^{k} c_i x_i = t$ with $t\in \ZZ_N$. Then the number of solutions in $S \subseteq \ZZ_N$ is
    $$
    \EE_{\xi\in \ZZ_N}\bigb{e_N\bigp{t\xi} \prod_{i=1}^{k} \widehat{S}(c_i\xi)}.
    $$
\end{lemma}

\begin{lemma}[Energy control]\label{lem:energy}
Let $S\subset \ZZ_N$ be a $B_h$ set for $h \ge 2$. Then
\[
    \|\widehat S\|_{L^{2h}} \le (h!)^{1/(2h)}\sqrt{\abs{S}}.
\]
\end{lemma}
\begin{proof}
By Parseval,
\[
    \mathbb E_{\xi\in \ZZ_N}\abs{\widehat S(\xi)}^{2h}
    =\sum_{x\in \ZZ_N}r_h(x)^2.
\]
Since $S$ is $B_h$, all representations counted by $r_h(x)$ are permutations of a single multiset, so $r_h(x)\le h!$ for every $x$. Therefore
\[
    \|\widehat S\|_{L^{2h}}^{2h}
    =\sum_{x\in \ZZ_N}r_h(x)^2
    \le h!\sum_{x\in \ZZ_N}r_h(x)
    =h!\abs{S}^h.
\]
\end{proof}

\begin{lemma}[Solution bound]\label{lem:freeze}
    Let $S \subset \ZZ_N$ be a $B_h$ set for $h \ge 2$. Consider any equation $\sum_{i=1}^{2h-1} c_i x_i = t$ with $t\in \ZZ_N$ and $c_1 \cdots c_{2h-1}$ coprime to $N$. Then the number of solutions in $S$ is $O\bigp{N^{1-\frac{1}{2h}}}$.
\end{lemma}
\begin{proof}
   By \Cref{lem:inversioncount}, the number of solutions in $S$ is
\begin{align*}
   \EE_{\xi\in \ZZ_N}\bigb{e_N\bigp{t\xi} \prod_{i=1}^{2h-1} \widehat{S}(c_i\xi)} &\le \prod_{i=1}^{2h-1} \bigp{\EE_{\xi\in \ZZ_N}\abs{\widehat{S}(c_i\xi)}^{2h-1}}^{\frac{1}{2h-1}}\\ &\le \prod_{i=1}^{2h-1} \bigp{\EE_{\xi\in \ZZ_N}\abs{\widehat{S}(c_i\xi)}^{2h}}^{\frac{1}{2h}}
\end{align*}
where both inequalities are by convexity. \Cref{lem:energy} then gives that this is $O\bigp{\abs{S}^{\frac{2h-1}{2}}} = O\bigp{N^{1-\frac{1}{2h}}}$, where the coprime assumption is used to guarantee that multiplication by $c_i$ is an automorphism of $\ZZ_N$ for every $i$.
\end{proof}

The result below establishes the key pseudorandomness of extremal $B_h$ sets. In the case of $h=2$ and working over $[N]$ this was already observed in \cite{exsidon}, with a proof using van der Corput differencing. Our proof partly mirrors that of the $C_hN^{1/h} + O_h(1)$ upper bound on the size of $B_h$ subsets of $\ZZ_N$, in particular the consideration of mixed sums and differences. It can essentially be seen as the observation that tightness in the upper bound forces pseudorandomness. Like the pseudorandomness of extremal Sidon sets established in \cite{exsidon}, there are likely other applications beyond the one presented here.

\pseud*

\begin{proof}
Denote $r \coloneqq \ceil{h/2}$ and $s \coloneqq \floor{h/2}$.
\begin{align*}
    \abs{\widehat{S}(\xi)}^h = \abs{\widehat{S}(\xi)^r \overline{\widehat{S}(\xi)}^s} = \abs{\sum_{\substack{x_1,\dots,x_r \in S \\ y_1,\dots,y_s \in S}} \ep{\xi \bigp{\sum_{i=1}^r x_i - \sum_{j=1}^s y_j}}},
\end{align*}
where the sum is over ordered tuples. Consider splitting this sum in two, one term (denoted $A$) corresponding to $x_1,\dots,x_r,y_1,\dots,y_s \in S$ all distinct and the other (denoted $B$) summing over tuples with at least one repeated coordinate. $B$ is over only $O_h(\abs{S}^{h-1})$ terms, so it is a lower order contribution. For $A$, define
$$
D \coloneqq \set{t: t = \sum_{i=1}^r x_i - \sum_{j=1}^s y_j \text{ for distinct } x_1,\dots,x_r,y_1,\dots,y_s \in S},
$$
so that
$$
A = \sum_{t \in D} \abs{m(t)} \ep{\xi t}, \hspace{0.7em} m(t) \coloneqq \set{(x_1,\dots,x_r,y_1,\dots,y_s) \in S \text{ distinct}: \sum_{i=1}^r x_i - \sum_{j=1}^s y_j = t}.
$$

For any $t \in D$, consider two tuples $(x_1,\dots,x_r,y_1,\dots,y_s)$ and $(x'_1,\dots,x'_r,y'_1,\dots,y'_s)$ in $m(t)$.
Then rearranging gives
$$
\sum_{i=1}^r x_i + \sum_{j=1}^s y'_j = \sum_{i=1}^r x'_i + \sum_{j=1}^s y_j.
$$
Because $S$ is $B_h$, we have equality of the multisets
$$
\set{x_1,\dots,x_r,y'_1,\dots,y'_s} = \set{x'_1,\dots,x'_r,y_1,\dots,y_s}.
$$
But since both tuples have only distinct elements we actually have equality of the sets
$$
\set{x_1,\dots,x_r} = \set{x'_1,\dots,x'_r}, \quad \set{y_1,\dots,y_s} = \set{y'_1,\dots,y'_s}.
$$
We thus have a unique set of $r$ elements along with a unique disjoint set of $s$ elements identified by any $t \in D$, so $\abs{m(t)} = r!s!$. Thus $A = r!s! \sum_{t \in D} \ep{\xi t}$. Now by orthogonality of characters,
$$
\abs{\sum_{t \in D} \ep{\xi t}} = \abs{\sum_{t \notin D} \ep{\xi t}} \le N - \abs{D} \le N - \frac{1}{r! s!}\abs{S}(\abs{S}-1) \cdots (\abs{S} - h + 1).
$$
The last inequality is by counting tuples $(x_1,\dots,x_r,y_1,\dots,y_s) \in S$ of distinct elements and then dividing out by the number that give the same sum ($r!s!$ by the same reasoning above). Thus, in total,
\begin{align*}
    \abs{\widehat{S}(\xi)}^h \le \abs{A} + \abs{B} &= r!s! N - \abs{S}\bigp{\abs{S}-1} \cdots (\abs{S} - h + 1) + O_h\bigp{\abs{S}^{h-1}}\\
    &= r!s!N - \abs{S}^h + O_h\bigp{\abs{S}^{h-1}}.
\end{align*}
\end{proof}

\section{Linear equations in $B_h$ sets}\label{sec:lineqbh}

With the ingredients in \Cref{sec:ingredients}, the first main result can be proved. Let us first make explicit the constant $D_{h,k}$ guaranteed by the result. Let $\lambda_{h,k}\in(0,1)$ be the unique solution to
$$
\lambda^{k-2h}=\bigp{\frac{C_h^h}{h!}}^h(1-\lambda)^{k-h}.
$$
Then $D_{h,k} \coloneqq C_h(1-\lambda_{h,k})^{1/h}$.

\begin{proof}[Proof of \Cref{thm:solving}]
Fix $D\in(D_{h,k},C_h)$, and suppose $\alpha \coloneqq \abs{S}N^{-1/h} \ge D$. Define
$$
\delta\coloneqq \bigp{C_h^h-D^h}^{1/h}, \quad \delta_{h,k} \coloneqq \bigp{C_h^h-D_{h,k}^h}^{1/h}.
$$
Since $D>D_{h,k}$, equivalently $\delta<\delta_{h,k}$, it holds by the definition of $D_{h,k}$ that
$$
\gamma\coloneqq D^h\bigp{D^{k-h}-h!\delta^{k-2h}} > D^h\bigp{D_{h,k}^{k-h}-h!\delta_{h,k}^{k-2h}} = 0.
$$
For every nonzero $\xi\in \ZZ_N$, \Cref{prop:psd} gives
$$
\abs{\widehat S(\xi)}\le (\delta+o(1))N^{1/h}.
$$
Let $T$ be the number of solutions $x_1,\dots,x_k \in S$ to $E$. By \Cref{lem:inversioncount},
$$
T=\frac{\abs{S}^k}{N}+R,
\quad
R\coloneqq \frac1 N\sum_{0 \neq \xi \in \ZZ_N}\widehat S(c_1\xi)\cdots \widehat S(c_k\xi).
$$
Pulling out $k-2h$ nonzero-frequency factors, applying convexity to the remaining $2h$ factors, and invoking \Cref{lem:energy},
\begin{align*}
\abs{R}
&\le \bigp{(\delta+o(1))N^{1/h}}^{k-2h}
\prod_{j=k-2h+1}^k
\bigp{\EE_{\xi\in\ZZ_N}\abs{\widehat S(c_j\xi)}^{2h}}^{1/(2h)} \\
&= \bigp{\delta+o(1)}^{k-2h}N^{\frac{k-2h}{h}}
\|\widehat S\|_{L^{2h}}^{2h} \\
&\le h!\bigp{\delta+o(1)}^{k-2h}\abs{S}^h N^{\frac{k-2h}{h}},
\end{align*}
where the coprime assumption is used to make multiplication by each $c_j$ a bijection. Thus
\begin{align*}
T
&\ge \alpha^k N^{\frac{k}{h}-1}
-h!\bigp{\delta+o(1)}^{k-2h}\alpha^h N^{\frac{k}{h}-1} \ge \bigp{\gamma-o(1)}N^{\frac{k}{h}-1}.
\end{align*}
In particular, $T\ge \frac{\gamma}{2}N^{\frac{k}{h}-1}$ for sufficiently large $N$.

It remains to count solutions without pairwise distinct coordinates. Such a solution has $x_i=x_{j}$ for some $i \ne j$. Fix this common value and $k-2h-1$ of the remaining variables, leaving $2h-1$ unfixed (recall $k \ge 2h+1$). The remaining $2h-1$ variables satisfy a nontrivial linear equation, so by \Cref{lem:freeze} there are $O_E(N^{1-\frac{1}{2h}})$ choices for them. Thus the total number of such solutions is
$$
O_E\bigp{\abs{S}^{k-2h}N^{1-\frac{1}{2h}}}
=O_E\bigp{N^{\frac{k}{h}-1-\frac{1}{2h}}}.
$$
For sufficiently large $N$, this is smaller than $T$, so $S$ contains a solution to $E$ with pairwise distinct coordinates, a contradiction. Therefore $\abs{S}<DN^{1/h}$ for every fixed $D>D_{h,k}$ and all sufficiently large $N$, which is equivalent to
$$
\abs{S}\le \bigp{D_{h,k} +o(1)}N^{1/h}.
$$
\end{proof}

The second main result needs an additional supersaturation input, but will not use \Cref{prop:psd}. Here is a supersaturated strengthening of \Cref{thm:qualvar}. 

\begin{theorem}[Varnavides \cite{varsat}]\label{thm:quantvar}
    Let $E: \sum_{i=1}^k c_i x_i = 0$ be such that $k\ge 3$ and
$\sum_i c_i =0$. For every $\eps>0$ there exists $\delta=\delta(\eps,E)>0$ such that the following holds. For any positive integer $N$ and any $A\subset \ZZ_N$ with size at least $\eps N$, there exists at least $\delta N^{k-1}$ solutions to $E$ in $A$.
\end{theorem}

See \cite{kosciuszkoinvariant} for a quantitative version. Recall \Cref{def:subdivision} of an $h$-subdivision. The $B_h$ property gives that the $h$-fold sumset of $S$ is dense, so Varnavides guarantees many solutions in the sumset that can be expanded to subdivision solutions.

\begin{proof}[Proof of \Cref{thm:sublin}]
Let \(m:=k/h \ge 3\), so after reindexing we may write
\[
E:\quad d_1(x_{1,1}+\cdots+x_{1,h})+\cdots+d_m(x_{m,1}+\cdots+x_{m,h})=0,
\]
with \(d_1,\dots,d_m\in \mathbb Z\setminus\{0\}\). Furthermore,
\[
d_1+\cdots+d_m=0
\]
since \(E\) is translation invariant.

Suppose for contradiction that \(|S| = \Omega\bigp{N^{1/h}}\). Define
\[
\sigma(T):=\sum_{x\in T}x, \quad 
A:=\set{\sigma(T): \ T\in \binom{S}{h}}.
\]
Note that $\sigma$ is injective on \(\binom{S}{h}\) by the $B_h$ property of $S$, so $|A|=\binom{\abs{S}}{h} = \Omega\bigp{\abs{S}^h} = \Omega(N)$. Then by \Cref{thm:quantvar} for the translation-invariant equation
\begin{equation}\label{eq:grouped}
    d_1a_1+\cdots+d_ma_m=0,
\end{equation}
the set \(A\) contains \(\Omega(N^{m-1})\) solutions \((a_1,\dots,a_m)\in A^m\).

Write each \(a_j\) uniquely as \(a_j=\sigma(T_j)\) with \(T_j\in \binom{S}{h}\). Call such a solution \emph{bad} if \(T_i\cap T_j\neq\varnothing\) for some \(i\neq j\). Fixing an intersecting pair \((T_i,T_j)\) costs \(O(\abs{S}^{2h-1})\) choices and fixes two variables of \eqref{eq:grouped}. Choosing \(m-3\) other variables costs $O(N^{m-3})$. For sufficiently large $N$, the last free variable in \eqref{eq:grouped} has a nonzero coefficient. This leaves only $O_E(1)$ choices for the last variable after the previous choices have been made. Therefore the number of bad solutions is
\[
O\bigl(\abs{S}^{2h-1}N^{m-3}\bigr)
=O\bigl(N^{m-1-1/h}\bigr)
=o(N^{m-1}).
\]
Hence some solution has \(T_1,\dots,T_m\) pairwise disjoint for sufficiently large $N$.

Now, for $1 \le j \le m$, write
\[
T_j=\set{x_{j,1},\dots,x_{j,h}}.
\]
Then
\[
d_1(x_{1,1}+\cdots+x_{1,h})+\cdots+d_m(x_{m,1}+\cdots+x_{m,h})=0,
\]
and all \(mh=k\) coordinates are pairwise distinct, contradicting the hypothesis on \(S\). Thus \(|S|=o\bigp{N^{1/h}}\).

\end{proof}

\section{Sidon dichotomy}\label{sec:sidondichotomy}

We will need a supersaturated strengthening of \Cref{thm:qualprend}. The result below was originally stated for $[N]$ instead of $\ZZ_N$, so we transfer to $\ZZ_N$.

\begin{theorem}[\cite{densesidon}*{Thm.~1.1}]\label{thm:prendiville}
    Let $E: \sum_{i=1}^k c_i x_i = 0$ be such that $k\ge 5$ and
$\sum_{i=1}^k c_i =0$. For every $\eps>0$ there exists $\sigma=\sigma(\eps,E)>0$ such that the following holds. For any positive integer $N$ and any Sidon $A\subset \ZZ_N$ with size at least $\eps \sqrt{N}$, there exists at least $\sigma N^{k/2-1}$ solutions to $E$ in $A$.
\end{theorem}

\begin{proof}
    $A$ is in bijection with a subset $B \subset [N]$ by identifying $\ZZ_N$ with integers mod $N$. Then $\abs{A} = \abs{B}$ and $B$ is Sidon, as any Sidon violation in $B$ would yield one in $A$ after reducing the equation mod $N$. It then holds that the energy
    $$
    \abs{\set{(x_1,x_2,x_3,x_4) \in B^4: x_1 + x_2 = x_3 + x_4}}
    $$
    over the integers is at most $2\abs{B}^2$. Applying \cite{densesidon}*{Thm.~1.1} with $\eta=0$ and taking $\sigma(\eps,E)$ small enough then guarantees at least $\sigma N^{k/2-1}$ solutions to $E$ in $B$. But this implies at least that many solutions in $A$, as any solution over the integers is also a solution mod $N$.
\end{proof}

The proof below follows the strategy of \cite{chromthresh}. Suppose $A \subset \ZZ_N$ is Sidon with $\abs{A} \ge \eps \sqrt{N}$. The key mechanism is to consider the Bohr set generated by the large spectrum of $A$. If $A$ has somewhat large intersection with this Bohr set, then those elements can be combined with the many solutions to the subequation on $I$ guaranteed by \Cref{thm:prendiville} to find a solution to $E$. The assumption that $\abs{I} \ge 5$ is used here both to invoke \Cref{thm:prendiville} and so that the $L^4$ Fourier control of \Cref{lem:energy} can be used to bound an error term from frequencies outside the large spectrum. If $A$ has small intersection with the Bohr set, then the large spectrum may be used to construct a small coloring. The number of colors is a function both of the intersection size and the number of frequencies in the large spectrum, so both need to be bounded. $L^4$ Fourier control is used again here, essentially as a replacement to Parseval to guarantee the large spectrum has bounded cardinality.

\structchar*
\begin{proof}
Let $m:=\abs{I}\ge 5$. By the universal upper bound on the size of any Sidon subset of $\ZZ_N$, since $C_2 = 1$,
$$
\abs{A} \le (1+o(1))\sqrt{N}.
$$ 
By taking $C$ large enough depending on $E$ it suffices to handle only the cases where $\abs{A} \le 2\sqrt{N}$ and all coefficients are nonzero in $\ZZ_N$.

Assume $\abs{A} \ge \eps \sqrt{N}$, so we aim to either find a pairwise distinct solution to $E$ or color $(\ZZ_N,A)$ with at most $C$ colors. Set
\[
\phi_A(\xi):=\frac{1}{2\sqrt{N}}\widehat{1_A}(\xi).
\]
Then for any $\xi \in \ZZ_N$, $\abs{\phi_A(\xi)}\le \frac{\abs{A}}{2\sqrt{N}}\le 1$. Crucially, by \Cref{lem:energy},
\[
\sum_{\xi\in \ZZ_N} \abs{\phi_A(\xi)}^4
 = \frac{1}{16N^2}\sum_{\xi\in \ZZ_N}\abs{\widehat{1_A}(\xi)}^4
 \le \frac{2\abs{A}^2}{16N}
 \le 1.
\]

Let us quickly handle when $I=[k]$, so $m=k$. By \Cref{thm:prendiville}, there exists
$\sigma=\sigma(\varepsilon,E)>0$ such that $A$ contains at least $\sigma N^{k/2-1}$ solutions to $E$. Call a solution \textit{bad} if $x_i=x_{j}$ for $i \neq j \in [k]$. After imposing
this relation, fix all but three other variables (using that $m \ge 5$). \Cref{lem:freeze} with $h=2$ gives at most $O(N^{3/4})$ remaining choices for these three variables. Multiplying back in the number of choices for the $k-5$ fixed variables, $x_i = x_{j}$, and $(i,j) \in [k]^2$, the total number of bad solutions is
\[
O_E\bigl(\abs{A}^{k-4}N^{3/4}\bigr)=O_E\bigl(N^{k/2-5/4}\bigr).
\]
For $N$ sufficiently large (depending only on $\varepsilon$ and $E$), this is smaller than
$\sigma N^{k/2-1}$, so $A$ contains a solution to $E$ with pairwise distinct coordinates. If $N$ is bounded in terms of $\varepsilon$ and $E$, then taking $C$ to be at least this bound gives
\[
\chi\bigl(\ZZ_N,A\bigr)\le N \le C.
\]

Suppose now $I \neq [k]$ and write $J:=[k]\setminus I$. Consider the translation-invariant subequation
\[
E_1:\ \sum_{i\in I} c_i x_i = 0.
\]
By \Cref{thm:prendiville}, there exists $\sigma=\sigma(\varepsilon,E_1)>0$ such that
$A$ contains at least $\sigma N^{m/2-1}$ solutions to $E_1$. Set
\[
\delta:=2^{-m}\sigma,
\quad \nu:=\frac{\delta}{6},
\quad \mathcal E:=\{\xi\in \ZZ_N:\abs{\phi_A(\xi)}\ge \nu\}.
\]
Since $\sum_{\xi}\abs{\phi_A(\xi)}^4\le 1$, we have
\[
\abs{\mathcal E}\le \nu^{-4}.
\]
Choose $s\in I$ arbitrarily, and set
\[
\Gamma:=c_s^{-1}\mathcal E=\{\eta\in \ZZ_N:c_s\eta\in \mathcal E\}.
\]
Then since multiplication by $c_s$ is a bijection on $\ZZ_N$, $\abs{\Gamma}=\abs{\mathcal E}\le \nu^{-4}$. For $x\in \ZZ_N$, write
\[
\normT{x}:=\left\lVert \frac{x}{N}\right\rVert_{\T}
=\min_{n\in \mathbb Z}\abs{\frac{x}{N}-n}.
\]
Now set
\[
D:=\sum_{i\in J}\abs{c_i},
\quad \rho:=\frac{1}{6\pi}D^{-1}\nu^4\delta,
\]
and consider the Bohr set
\[
B=B(\Gamma,\rho):=\{x\in \ZZ_N:\normT{\xi x}\le \rho\ \text{for all }\xi\in \Gamma\}.
\]

\textbf{Case $\abs{A\cap B}\ge k$.} We may choose pairwise distinct elements
$\{x_i:i\in J\}$ with $x_i\in A\cap B$ for each $i\in J$. Set
\[
y:=-\sum_{i\in J} c_i x_i.
\]
Let $N_1,N_2$ denote the number of solutions in $A$ to
\[
\sum_{i\in I} c_i x_i = 0
\quad\text{and}\quad
\sum_{i\in I} c_i x_i = y
\]
respectively. By construction, $N_1\ge \sigma N^{m/2-1}=2^m\delta N^{m/2-1}$. Moreover, \Cref{lem:inversioncount} gives
\[
N_1 = 2^m N^{m/2-1}\sum_{\xi\in \ZZ_N}\prod_{i\in I}\phi_A(c_i\xi),
\quad
N_2 = 2^m N^{m/2-1}\sum_{\xi\in \ZZ_N}\prod_{i\in I}\phi_A(c_i\xi)\ep{y\xi}.
\]
Thus
\[
\frac{\abs{N_1-N_2}}{2^m N^{m/2-1}}
\le \sum_{\xi\in \ZZ_N}\prod_{i\in I}\abs{\phi_A(c_i\xi)}\abs{1-\ep{y\xi}}.
\]
We split the sum according to whether $\xi\in \Gamma$.

If $\xi\in \Gamma$, then using that $\set{x_i: i \in J} \subseteq B$,
\[
\abs{1-\ep{y\xi}}
\le 2\pi\normT{y\xi}
\le 2\pi\sum_{i\in J}\normT{c_i x_i\xi}
\le 2\pi\sum_{i\in J}\abs{c_i}\normT{x_i\xi}
\le 2\pi D\rho
\le \frac13\nu^4\delta.
\]
Combining this with the fact that $\abs{\phi_A}\le 1$,
\[
\sum_{\xi\in \Gamma}\prod_{i\in I}\abs{\phi_A(c_i\xi)}\abs{1-\ep{y\xi}}
\le \abs{\Gamma}\cdot \frac13\nu^4\delta
\le \frac13\delta.
\]

If $\xi\notin \Gamma$, then $\abs{\phi_A(c_s\xi)}<\nu$. Also $\abs{1-\ep{y\xi}}\le 2$, so
\[
\sum_{\xi\notin \Gamma}\prod_{i\in I}\abs{\phi_A(c_i\xi)}\abs{1-\ep{y\xi}}
\le 2\nu\sum_{\xi\in \ZZ_N}\prod_{i\in I\setminus\{s\}}\abs{\phi_A(c_i\xi)}.
\]
By H\"older with exponent $m-1$,
\[
\sum_{\xi\in \ZZ_N}\prod_{i\in I\setminus\{s\}}\abs{\phi_A(c_i\xi)}
\le \prod_{i\in I\setminus\{s\}}\left(\sum_{\xi\in \ZZ_N}\abs{\phi_A(c_i\xi)}^{m-1}\right)^{1/(m-1)}.
\]
Recall now that $m-1\ge 4$, $\abs{\phi_A}\le 1$, and multiplication by $c_i$ is a bijection of $\ZZ_N$ for every $i$. Then each factor is at most
\[
\left(\sum_{\xi\in \ZZ_N}\abs{\phi_A(\xi)}^4\right)^{1/(m-1)}\le 1,
\]
so the whole product is at most $1$. Therefore,
\[
\sum_{\xi\notin \Gamma}\prod_{i\in I}\abs{\phi_A(c_i\xi)}\abs{1-\ep{y\xi}}
\le 2\nu \le \frac13\delta.
\]

Summing both pieces,
\[
\abs{N_1-N_2}\le \frac23\cdot 2^m\delta N^{m/2-1}
=\frac23\sigma N^{m/2-1},
\]
so $N_2\ge \frac13\sigma N^{m/2-1}$.

Call a solution to $\sum_{i\in I} c_i x_i=y$ in $A$ \textit{bad} if either $x_i=x_{i'}$ or $x_i=x_j$ for $i \neq i' \in I, j \in J$. Just as in the case where $I = [k]$, since $\{x_j:j\in J\}$ are considered fixed, using \Cref{lem:freeze} each such family of bad solutions is of size at most
\[
O_E\bigl(\abs{A}^{m-4}N^{3/4}\bigr)=O_E\bigl(N^{m/2-5/4}\bigr).
\]
Since there are only $O_E(1)$ bad relations of the above form, the total number of bad solutions is
$O_E(N^{m/2-5/4})$.

For $N > N_0(\eps,E)$, this is smaller than $\frac13\sigma N^{m/2-1}$. Then there exists a solution to
$\sum_{i\in I} c_i x_i=y$ in $A$ with pairwise distinct coordinates, all distinct from $\{x_j:j\in J\}$. Adding
back the chosen $(x_j)_{j\in J}$, we obtain a solution to $E$ in $A$ with pairwise distinct coordinates. This completes the case where $\abs{A \cap B} \ge k$ and $N > N_0$.

Enlarge $C$ further so that $C \ge N_0$. If $N\le N_0$, then $\chi(\ZZ_N,A)\le N\le C$. Thus for the remaining case we assume $N>N_0$.

\textbf{Case $\abs{A \cap B} < k$.}
 Let
\[
M:=\left\lceil 2\rho^{-1}\right\rceil,
\]
and partition the unit circle $S^1$ into arcs $I_0,\dots,I_{M-1}$ of equal length $2\pi/M$. Define a map
$\kappa:\ZZ_N\to [M]^\Gamma$ as follows: for each $u\in \ZZ_N$ and each $\xi\in \Gamma$, let $\kappa_\xi(u)$ be the
unique index such that $\ep{\xi u}\in I_{\kappa_\xi(u)}$, and set
\[
\kappa(u):=(\kappa_\xi(u))_{\xi\in \Gamma}.
\]
Then the image of $\kappa$ satisfies
\[
\abs{\operatorname{im}(\kappa)}\le M^{\abs{\Gamma}}
\le \left\lceil 2\rho^{-1}\right\rceil^{\nu^{-4}}.
\]
For each $a\in \operatorname{im}(\kappa)$, let $V_a:=\kappa^{-1}(a)$. These sets partition $\ZZ_N$. If $v_2$ is an
in-neighbor of $v_1$ in the induced subgraph $(\ZZ_N,A)[V_a]$, then $v_1-v_2\in A$, and for every
$\xi\in \Gamma$ we also have
\[
\ep{\xi v_1},\ep{\xi v_2}\in I_{a_\xi}
\quad\Longrightarrow\quad
\normT{\xi(v_1-v_2)}\le \frac{2}{M}\le \rho.
\]
Thus $v_1-v_2\in A\cap B$. Since $\abs{A \cap B} < k$, each vertex of $(\ZZ_N,A)[V_a]$ has at most $k-1$
in-neighbors and at most $k-1$ out-neighbors. Therefore the underlying undirected graph has maximum
degree at most $2(k-1)$, so
\[
\chi\bigl((\ZZ_N,A)[V_a]\bigr)\le 2k-1.
\]
Summing over all parts,
\[
\chi\bigl(\ZZ_N,A\bigr)
\le \sum_{a\in \operatorname{im}(\kappa)} \chi\bigl((\ZZ_N,A)[V_a]\bigr)
\le (2k-1)\left\lceil 2\rho^{-1}\right\rceil^{\nu^{-4}}.
\]
Taking $C$ larger than this, which still only depends on $\eps$ and $E$, completes the proof.
\end{proof}

\section{Sidon constructions}\label{sec:sidonconstructions}

Two examples of extremal Sidon sets are 
\begin{equation}\label{eq:sidonex}
    \set{(x,x) \in \FF_p^\times \times \FF_p: x \in \FF_p^\times} \text{ and } \set{(x,x^2) \in \FF_p^2: x \in \FF_p},
\end{equation}
both for prime $p$. These are due to Spence (see \cite{spence}*{Thm.~4.4}) and Erd\H{o}s and Turán \cite{erdosturansidon} respectively. In general, Sidon subsets of size $\Omega\bigp{\sqrt{N}}$ in abelian groups of size $N$ are hard to come by. All known examples seem in some sense to come from the same source, as observed and discussed in \cite{apparentstructure}. This makes our task harder, as all of our constructions must then roughly be of this form unless we are to disprove conjectures in \cite{apparentstructure}.

Before diving in, let us discuss a relevant strengthening of the Sidon property introduced by Lazebnik and Verstraëte in \cite{lazebnikverstraete}. For $N$ coprime to all numbers in $[k]$, $S \subset \ZZ_N$ is \textit{$k$-fold Sidon} if it is free of nontrivial solutions to
$$
a_1 x_1 + a_2 x_2 + a_3 x_3 + a_4 x_4 = 0
$$
whenever $a_1,a_2,a_3,a_4 \in \set{-k,-k+1,\dots,0,\dots,k-1,k}$ and $a_1 + a_2 + a_3 + a_4 = 0$. A $1$-fold Sidon set is just Sidon. \cite{lazebnikverstraete} constructs $2$-fold Sidon sets in $\ZZ_N$ of size $\frac 1 2 \sqrt{N}-3$ for $N = 2^{2^{i+1}}+2^{2^i}+1$ for positive integer $i$. We will use this construction in \Cref{thm:exeq}. It remains an open problem to construct any $3$-fold Sidon subset of $\ZZ_N$ of size $\Omega\bigp{\sqrt{N}}$ for infinitely many $N$. This is indicative of the difficulty of finding large Sidon sets that are also free of other linear equations, and throughout this section we will further be asking for them to generate Cayley graphs with unbounded chromatic number. It seems likely the constructions pursued below are easier to strengthen over non-cyclic abelian groups, but we do not consider other groups here.

\subsection{2-fold Sidon}\label{sec:twofoldcon}

\exeq*

\begin{proof}
    Call $D \subset \ZZ_N$ a \textit{perfect difference} set if every nonzero element of $\ZZ_N$ can be written uniquely as the difference of two elements in $D$. For positive integer $i$ let $t \coloneqq 2^i$, $N \coloneqq 2^{2t} + 2^t + 1$. In \cite{singer}, Singer constructs perfect difference sets $D$ with $D = 2D$ and $\abs{D} = 2^t + 1$ living inside $\ZZ_N$. Lazebnik and Verstraëte then use these in \cite{lazebnikverstraete}*{Thm.~2.5} to construct 2-fold Sidon $S \subset D$ with the property $S \cap 2S = \emptyset$. Explicitly $S$ consists of every other element in the cyclic decomposition of $D$ given by the $\times 2$-map, but the only properties needed from this construction are those stated above. Being $2$-fold Sidon forbids nontrivial solutions to the first and third equations in the theorem, but the same argument in fact forbids solutions to the second. Indeed, consider any solution $x_1,x_2,x_3,x_4 \in S$ to the second equation. Since $2S \subset 2D = D$, it holds that $2x_1 \in D$. Then since $D$ is Sidon, it holds that $\set{2x_1,x_4} = \set{x_2,x_3}$. In particular $2x_1 \in S$, which contradicts $x_1 \in S$ as $S \cap 2S = \emptyset$.

    It remains to verify that $\chi(\ZZ_N,S) = \Omega(N^{1/8})$. By the construction of \cite{lazebnikverstraete}, $\abs{S} \ge 2^{t-1}-2$ and $0 \in D$, so there is
a set \(T\subset \mathbb Z_N\) with \(|T|\le 4\) such that
\[
D\setminus\{0\}\subseteq S\cup 2S\cup T
\]
and $S,2S,T$ are pairwise disjoint. Since multiplication by \(2\) is an automorphism of \(\mathbb Z_N\), the underlying undirected Cayley graphs of $(\ZZ_N,S)$ and $(\ZZ_N,2S)$ are isomorphic. The underlying undirected Cayley graph of $(\ZZ_N,T)$ has degree at most $8$. It follows that
\[
\chi(\ZZ_N,D)\le \chi(\ZZ_N,S)\chi(\ZZ_N,2S)\chi(\ZZ_N,T)\le 9\chi(\ZZ_N,S)^2.
\]
Let us now lower bound $\chi(\ZZ_N,D)$. Denote \(q:=\abs{D \setminus \set{0}}=2^t\), so \(N=q^2+q+1\). Let \(I\) be an independent set in $(\ZZ_N,D)$, and
for \(y\notin I\) let $d_y$ denote the number of edges from $I$ to $y$, that is
\[
d_y:=\bigl|\{x\in I:\ y-x\in D\setminus\{0\}\}\bigr|.
\]
Then \(\sum_{y\notin I} d_y=q|I|\). Note that because $D$ is a perfect difference set every pair of distinct vertices $x,y \in I$ has a unique common out-neighbor outside of $I$. Indeed, $x - y$ has a unique representation as $d_1-d_2$ for $d_1,d_2 \in D$. Their common out-neighbor is then $x+d_2 = y+d_1$, which is outside $I$ since $I$ is an independent set. Using this fact along with convexity,
$$
{\abs{I} \choose 2} = \sum_{y \notin I} {d_y \choose 2} \ge (N-\abs{I}){\frac{q\abs{I}}{N-\abs{I}} \choose 2} = \frac 1 2 \bigp{\frac{q^2\abs{I}^2}{N-\abs{I}} - q\abs{I}}.
$$
Using \(N=q^2+q+1\), after a fair amount of algebra this simplifies to
\[
(|I|-1)^2\le q^3.
\]
Then \(\alpha(\ZZ_N,D)\le q^{3/2}+1=O(N^{3/4})\), and therefore
\[
\chi(\ZZ_N,D)\ge \frac{N}{q^{3/2}+1}=\Omega(N^{1/4}).
\]
    
\end{proof}

\subsection{Parabola}\label{sec:parabolacon}

The parabola construction used for both \Cref{thm:parabola3} and \Cref{thm:parabola4} is inspired by the Erd\H{o}s--Turán construction living in $\FF_p^2$ introduced above \eqref{eq:sidonex}, but compressed to live in a cyclic group. It actually avoids many more equations in addition to $E$, as shown explicitly in the proofs below.

\parabolathree*
\begin{proof}
    Fix any integer $M$ such that $M \ge 5$, $\sum_i \abs{c_i} < M$, and $\prod_i c_i$ divides $M$. Take $p$ to be any prime larger than $M^2$. Define $m \coloneqq Mp+1$ and
    $$
    A \coloneqq \set{(z,z^2) \in \ZZ_m \times \FF_p: 1 \le z \le p-1}.
    $$
    Then if $N \coloneqq mp$, $A \subset \ZZ_{N}$ and $\abs{A} = \Omega_M\bigp{\sqrt{N}}$. Furthermore, $N$ is coprime to $M$ and hence $c_1c_2c_3$.

    To show $A$ is Sidon, consider any violation $x_1,x_2,x_3,x_4 \in [p-1]$ satisfying
    $$
    x_1 + x_2 \equiv x_3 + x_4 \pmod m \quad \text{and} \quad
x_1^2+x_2^2 \equiv x_3^2 + x_4^2 \pmod p
    $$
    with $\set{x_1,x_2} \neq \set{x_3,x_4}$. Since $M \ge 5$, by definition of $m$ there is no torsion in the $\ZZ_m$ coordinate. Thus
    $$
    x_1 + x_2 = x_3 + x_4
    $$
    as integers. But then
    $$
    x_1 + x_2 \equiv x_3 + x_4 \pmod p \quad \text{and} \quad
x_1^2+x_2^2 \equiv x_3^2 + x_4^2 \pmod p.
    $$
    Since $\set{x_1,x_2} \neq \set{x_3,x_4}$, this forms a Sidon violation in the Erd\H{o}s--Turán Sidon construction \eqref{eq:sidonex}, a contradiction.

    To show $A$ is free of solutions to $E$, we will show the stronger statement that $A$ is free of nontrivial solutions to any equation of the form $a_1x_1 + a_2x_2 + a_3 x_3 = 0$ with
    \begin{center}
    \begin{enumerate*}[label=(\textit{\roman*})]
        \item $\abs{a_1} + \abs{a_2} + \abs{a_3} < M$, \quad
        \item $a_1 + a_2 + a_3 = 0$.
    \end{enumerate*}
    \end{center}
    Consider such a solution in $A$, that is $x_1,x_2,x_3 \in [p-1]$ satisfying
    \begin{equation*}\label{eq:collinear}
a_1x_1+a_2x_2+a_3x_3\equiv 0 \pmod m \quad \text{and} \quad
a_1x_1^2+a_2x_2^2+a_3x_3^2\equiv 0 \pmod p.
    \end{equation*}
    Again by definition of $m$ there is no torsion in the first coordinate, and the first equation gives
    \begin{equation}\label{eq:lin3}
        a_1x_1 + a_2x_2 + a_3x_3 = 0
    \end{equation}
    as integers.
If $z:=a_1x_1+a_2x_2$, then the second equation can be rewritten over $\FF_p$ as
\[
z^2+a_1a_2(x_1-x_2)^2 = z^2 \implies a_1a_2(x_1-x_2)^2 = 0,
\]
where we have used that $p > \abs{a_3}$. Since also $p > M^2 > \abs{a_1a_2}$, it follows that \(x_1\equiv x_2 \mod p\). As
\(x_1,x_2\in[p-1]\), this gives \(x_1=x_2\). Substituting into \eqref{eq:lin3} and using
\(a_1+a_2=-a_3\) gives \(x_3=x_1\). Then \(x_1=x_2=x_3\), so the solution is trivial.

    For the chromatic number, we unnecessarily use a point-line incidence bound. For prime $p$, a set $\cP$ of points, and a set $\cL$ of lines in $\FF_p^2$, let $I(\cP,\cL)$ denote the number of point-line incidences between $\cP$ and $\cL$. That is, the number of pairs consisting of $u \in \cP$ and $v \in \cL$ such that $u \in v$. \cite{cillincidence} established that
    \begin{equation}\label{eq:incidence}
        I(\cP,\cL) = \frac{\abs{\cP}\abs{\cL}}{p} + O\bigp{\sqrt{p\abs{\cP}\abs{\cL}}},
    \end{equation}
    with the upper bound having previously been established in \cite{vinszem}.
    Interestingly, Cilleruelo's proof in \cite{cillincidence} actually uses Sidon sets, but our connection to Sidon sets is a bit different (Cilleruelo works with Spence's Sidon construction instead of a parabola). Suppose there exists a proper coloring of $(\ZZ_N,A)$ with $c\sqrt{p}$ colors for some constant $c > 0$. Then restricting this coloring to the vertices
    $$
        (x,y) \in \set{0,1,\dots,p-1} \times \FF_p \subset \ZZ_m \times \FF_p
    $$
    gives an independent subset of size $I = c^{-1}p^{3/2}$. Split this independent set into two parts: $L$ consisting of all points with $x \le t$ and $R$ with $x > t$ where $t$ is chosen so that both are of size at least $I/3$ (recalling that $I = \omega(p)$ and the number of $(x,y)$ with a given $x$ is at most $p$). Then by independence there are no edges between $L$ and $R$ of the form $(x_\ell,y_\ell) \to (x_r,y_r)$ with $y_r-y_\ell = (x_r-x_\ell)^2$ in $\FF_p$, using that $x_r - x_\ell \in \set{1,\dots,p-1}$ so that the difference vector indeed lies in the generating set $A$.

    Let us rewrite
    $$
    y_r-y_\ell = (x_r-x_\ell)^2 \iff y_r - x_r^2 = -2x_\ell x_r + (y_\ell + x_\ell^2).
    $$
    Making the change of coordinates $z_r \coloneqq y_r - x_r^2$, this becomes
    $$
    z_r = -2x_\ell x_r + (y_\ell+x_\ell^2).
    $$
    Then $L$ and $R$ are in bijection with a set of lines $\cL$ and points $\cP$ respectively, defined by
    \begin{align*}
        &\cL \coloneqq \set{z = -2x_\ell x + (y_\ell + x_\ell^2): (x_\ell,y_\ell) \in L},\\
        &\cP \coloneqq \set{(x,z): x = x_r, z = y_r-x_r^2 \text{ for } (x_r,y_r) \in R}.
    \end{align*}
    By definition, the lack of edges between $L$ and $R$ implies that $I(\cP,\cL) = 0$, as any incidence would map to an edge. By \eqref{eq:incidence},
    $$
    \frac{\abs{\cP}\abs{\cL}}{p} = O\bigp{\sqrt{p\abs{\cP}\abs{\cL}}} \implies \frac1 3 I \le \sqrt{\abs{\cP}\abs{\cL}} = O\bigp{p^{3/2}},
    $$
    a contradiction if $c$ is taken small enough. Thus $\chi(\ZZ_N,A) = \Omega\bigp{\sqrt{p}} = \Omega_M(N^{1/4})$.
\end{proof}

\parabolafour*
\begin{proof}
    The construction and proof are essentially identical to that of \Cref{thm:parabola3}, with only a change in parameters and a slight twist in showing solution freeness. Again fix any integer $M$ such that $M \ge 5$, $\sum_i \abs{c_i} < M$, and $\prod_i c_i$ divides $M$. Choose $p$ to be any prime larger than $M^2$ such that $\frac{c_1c_2}{c_3c_4}$ is not a square in $\FF_p$.

    \begin{claim}
        Let $q \in \QQ$ be a non-square. Then there are infinitely many primes $p$ such that $q$ is not a square in $\FF_p$.
    \end{claim}
    \begin{proof}
        Let $q = a/b$ in lowest terms. Since $q$ is not a square, the integer $ab$ is also not a square. For every prime $p$ not dividing $b$,
        $$q = \frac{a}{b} = \frac{ab}{b^2}$$
        in $\mathbb{F}_p$, so $q$ is a square mod $p$ if and only if $ab$ is a square mod $p$. Since $ab$ is not a square as an integer, by \cite{ireland}*{Ch.~5, Thm.~3} there are infinitely many primes $p$ for which $ab$ is a quadratic nonresidue mod $p$ (not a square). Discarding the finitely many primes dividing $b$, we get infinitely many for which the reduction of $q$ mod $p$ is well-defined and not a square.
    \end{proof}

     As before, define $m \coloneqq Mp+1$ and
    $$
    A \coloneqq \set{(z,z^2) \in \ZZ_m \times \FF_p: 1 \le z \le p-1}.
    $$
    Then if $N \coloneqq mp$, $A \subset \ZZ_{N}$ and $\abs{A} = \Omega_M\bigp{\sqrt{N}}$. Furthermore, $N$ is coprime to $M$ and hence $c_1c_2c_3c_4$.

    The Sidon property and chromatic number lower bound follow exactly as in the proof of \Cref{thm:parabola3}, as the additional constraint on $p$ has no effect.

    To show $A$ is free of solutions to $E$, we show the stronger statement that $A$ is free of nontrivial solutions to any equation of the form $a_1x_1 + a_2x_2 = a_3x_3 + a_4x_4$ with
    \begin{center}
        \begin{enumerate*}[label=(\textit{\roman*})]
            \item $\abs{a_1} + \abs{a_2} + \abs{a_3} + \abs{a_4} < M$, \quad
            \item $a_1+a_2=a_3+a_4$, \quad
            \item $\frac{a_1a_2}{a_3a_4} = \frac{c_1c_2}{c_3c_4}$.
        \end{enumerate*}
    \end{center}
    Consider such a solution in $A$, that is $x_1,x_2,x_3,x_4 \in [p-1]$ satisfying
    \begin{equation*}\label{eq:modular}
        a_1x_1 + a_2x_2 \equiv a_3x_3 + a_4x_4 \pmod m \quad \text{and} \quad a_1x_1^2 + a_2x_2^2 \equiv a_3x_3^2 + a_4x_4^2 \pmod p.
    \end{equation*}
   Again by definition of $m$ there is no torsion in the first coordinate, and the first equation gives
    \begin{equation}\label{eq:lin}
        a_1x_1 + a_2x_2 = a_3x_3 + a_4x_4
    \end{equation}
    as integers.
    If $z \coloneqq a_1x_1 + a_2x_2$, then the second equation can be rewritten over $\FF_p$ as
    \begin{equation}\label{eq:quad}
    z^2 + a_1a_2(x_1-x_2)^2 = z^2 + a_3a_4(x_3-x_4)^2 \implies \frac{a_1a_2}{a_3a_4}(x_1-x_2)^2 = (x_3-x_4)^2,
    \end{equation}
    where we have used that $p > M^2 > \abs{a_3a_4}$. Since $\frac{c_1c_2}{c_3c_4} = \frac{a_1a_2}{a_3a_4}$ is not a square in $\FF_p$, \eqref{eq:quad} gives that $x_1 \equiv x_2$ and $x_3 \equiv x_4$ mod $p$. But again since they all lie in $[p-1]$ these are equalities. Now combining $a_1 + a_2 = a_3 + a_4$ with $\eqref{eq:lin}$ gives that either $a_1+a_2 = 0$ or $x_1 = x_2 = x_3 = x_4$, so any solution is trivial.
\end{proof}

\subsection{Random sprinkling}\label{sec:sprinklecon}

We turn now to \Cref{thm:twosum}, which constructs large Sidon sets free of any equation with no zero-sum subset of at least three coefficients, in particular non-translation-invariant ones. Given a large Sidon subset $S \subset \ZZ_N$ and some constant $M$ coprime to $N$, it is easy to construct a large Sidon subset in $\ZZ_{NM}$ free of solutions to many non-translation-invariant equations by taking
\begin{equation}\label{eq:sidprod}
    \set{(s,1) \in \ZZ_{N} \times \ZZ_M: s \in S}.
\end{equation}
This set is free of solutions to any equation of the form $\sum_i c_i x_i = 0$ if $\sum_i c_i \neq 0$ and $\sum_i \abs{c_i} < M$. But this generates a Cayley graph with chromatic number at most $3$ simply by taking the $\ZZ_M$ coordinate mod $3$. Thus our goal will be to sprinkle random edges in each level, just enough to break up independent sets but not so many as to create solutions. The difficulty is ensuring these sprinkled elements do not interact with the large Sidon set to create solutions, which is why we are unable to handle arbitrary equations with zero-sum subsets of three or four coefficients.

For positive integer $M$, let $\cL(M)$ denote the collection of equations $\sum_i c_ix_i = 0$ where $\sum_i \abs{c_i} < M$ and there does not exist $I\subseteq [k]$ with $\abs{I}\ge 3$ and $\sum_{i\in I} c_i=0$. The following random construction (the sprinkled edges) gives unbounded chromatic number.

\begin{lemma}\label{lem:randomsidon}
        For integer $M \ge 5$ and $N \ge N_0(M)$, there exists Sidon $R \subset \ZZ_N$ free of nontrivial solutions to any equation in $\cL(M)$. Furthermore $\chi(\ZZ_N, R) = \widetilde{\Omega}(N^{\frac{1}{2M}})$.
    \end{lemma}
    \begin{proof}
     Consider sampling $R$ by including each element with probability $q = N^{\frac 1 M - 1}$ independently. We argue that with high probability as $N \to \infty$ the desired properties hold. $N_0(M)$ may then be defined as the smallest such that these properties hold with probability at least $1/2$, say, for all $N \ge N_0(M)$.
     
     We first establish that $R$ is free of nontrivial solutions to any equation in $\cL = \cL(M)$. Note that any equation in $\cL$ that is translation invariant has the form $c(x_1-x_2) = 0$. After choosing $x_1$, there are then $O_{c}(1)$ choices for $x_2$, so the expected number of nontrivial solutions to such an equation in $R$ is $O_{c}(Nq^2) = O_c\bigp{N^{\frac 2 M - 1}} = o(1)$. Considering all possible values of $c$ for equations in $\cL$ gives $O_M(Nq^2) = o(1)$ expected nontrivial solutions, so with high probability $R$ is free of all of them. It thus suffices only to consider equations in $\cL$ that are not translation invariant.
    \begin{claim}\label{claim:transinvar}
        Consider an equation
            $$
            E: \sum_{i=1}^L c_i x_i = 0
            $$
        with $\sum_i c_i \neq 0$ and integer $k \le L$. The number of solutions $x_1,\dots,x_L \in \ZZ_N$ with exactly $k$ distinct variables is $O_{E}(N^{k-1})$.
    \end{claim}
    \begin{proof}
        For a solution $x_1,\dots,x_L$ with $k$ distinct, we may add the coefficients corresponding to equal variables to yield a new equation
        $$
        \sum_{i=1}^k d_i y_i = 0,
        $$
        where $y_1,\dots,y_k$ are distinct and $\sum_i d_i = \sum_i c_i \neq 0$. In particular one of the $d_i$ is nonzero. Assume without loss of generality it is $d_1$. Then after choosing $y_2,\dots,y_k$ ($N^{k-1}$ choices), there are $O_{d_1}(1)$ choices for $y_1$. Multiplying by the $O_L(1)$ ordered tuples $x_1,\dots,x_L$ with the same distinct $y_1,\dots,y_k$ gives the result.
    \end{proof}
    Consider $E \in \cL$ that is not translation invariant. Using the claim, where $L < M$, the expected number of solutions to $E$ in $R$ with $k$ distinct variables is at most
    $$
    O_{E}\bigp{N^{k-1}q^k} = O_{E}\bigp{N^{k-1 + k\bigp{\frac 1 M - 1}}} = O_{E}\bigp{N^{\frac k M - 1}} = O_{M}\bigp{N^{\frac{M-1}{M} - 1}}.
    $$
    Thus the expected number of solutions to $E$ is $O_M\bigp{N^{\frac{M-1}{M} - 1}}$, and so is the expected number of overall solutions to equations in $\cL$ that are not translation invariant. This vanishes as $N \to \infty$, so with high probability $R$ is free of all of these.

    Let us now establish the Sidon property. Note that any solution to $x_1 + x_2 = x_3 + x_4$ violating the Sidon property must have a nonzero coefficient after combining equal variables, like in \Cref{claim:transinvar}. By the same argument there are $O(N^{k-1})$ such solutions with $k$ distinct variables for $k \in [4]$. Thus the expected number surviving in $R$ is $O(N^{k-1}q^k) = o(1)$ for each such $k$ (using that $M \ge 5$), and with high probability $R$ is free of all Sidon violations.

    Finally, the fact that $\chi(\ZZ_N, R) = \widetilde{\Omega}(N^{\frac{1}{2M}})$ with high probability follows from standard spectral estimates for random Cayley graphs and Hoffman's ratio bound on the size of the largest independent set. See \cite{cayleycoloring}*{Prop.~2.7} for a statement that directly gives the desired bound.
    \end{proof}

Combining \eqref{eq:sidprod} with the lemma yields both near-maximal size and unbounded chromatic number. Just as in Theorems \ref{thm:parabola3} and \ref{thm:parabola4}, our construction actually avoids many more equations beyond $E$, as explicitly proved below.

\twosum*

\begin{proof}
    Denote $c \coloneqq \abs{\prod_i c_i}$. Fix any integer $M$ such that $M \ge 5$, $\sum_i \abs{c_i} < M$, and $(M,c) = 1$. Denote $d \coloneqq Mc$. Then since $1+d$ is coprime to $d^2$, by Dirichlet's theorem there exist infinitely many primes $p$ such that
    $$
    p \equiv 1+d \mod d^2 \implies p-1 = dr \text{ for } r \equiv 1 \mod d.
    $$
    Fix any such $p,r$ with $r = \frac{p-1}{d} \ge \max\bigp{M,N_0(M)}$, where $N_0(M)$ is as in \Cref{lem:randomsidon}.

    Let $R \subseteq \ZZ_{r}$ be as in \Cref{lem:randomsidon}. Since $r$ divides $p-1$, $\ZZ_r$ is isomorphic to a subgroup $H \le \FF_p^\times$. Fix any such isomorphism $\phi: \ZZ_r \to H \subset [p-1]$. Define
    $$
    P \coloneqq \set{(x,\phi(x),1) \in \ZZ_r \times \FF_p \times \ZZ_M: x \in \ZZ_r}, \quad Q \coloneqq \set{(u,0,0) \in \ZZ_r \times \FF_p \times \ZZ_M: u \in R}.
    $$
    Since $r \equiv 1 \mod d$ and $p > r \ge M$, it holds that $r,p,M$ are pairwise coprime and $N \coloneqq Mpr$ is coprime to $c$. Then $P \cup Q$ lives in the cyclic group $\ZZ_N$ for $N$ coprime to $c_1 \cdots c_k$ as desired. Furthermore,
    $$
    \abs{P \cup Q} \ge \abs{P} = r = \frac{\sqrt{r}}{\sqrt{Mp}}\sqrt{N} = \frac{\sqrt{p-1}}{\sqrt{dMp}}\sqrt{N} = \Omega_E\bigp{\sqrt{N}}.
    $$
    Note also that 
    $$
    \chi(\ZZ_N,P \cup Q) \ge \chi(\ZZ_r,R) = \widetilde{\Omega}\bigp{r^{\frac{1}{2M}}} = \widetilde{\Omega}_E\bigp{N^{\frac{1}{4M}}} = N^{\Omega_E(1)},
    $$ 
    since $(\ZZ_r,R)$ is induced on the vertex subset $(\ZZ_r,0,0)$ by $Q$. It thus remains to check that $P \cup Q$ is Sidon and free of solutions to $E$. We will in fact show that $P \cup Q$ is free of nontrivial solutions to any equation in $\cL = \cL(M)$, which includes $E$ by definition of $M$.

    To show that $P \cup Q$ is Sidon, consider any quadruple
    $$
    x_1 + x_2 = x_3 + x_4.
    $$
    Considering the third coordinate, both sides must have the same number of variables in $P$ and the same number in $Q$ ($M \ge 5$ so there is no torsion).
    
    If all variables lie in $P$ and form a Sidon violation, then their projections onto the first two coordinates form a Sidon violation in
    \begin{align*}
        \set{(x,\phi(x)) \in \ZZ_r \times \FF_p: x \in \ZZ_r} &\cong \set{(\phi(x),\phi(x)) \in H \times \FF_p: x \in \ZZ_r} \\
        & = \set{(z,z) \in H \times \FF_p: z \in H} \\
        &\subset \set{(z,z) \in \FF_p^\times \times \FF_p: z \in \FF_p^\times},
    \end{align*}
    where we have used that $H$ is a subgroup of $\FF_p^\times$. But this containing set is known to be Sidon \eqref{eq:sidonex}, a contradiction.

    Similarly, if all variables of a Sidon violation lie in $Q$ then their projection onto the first coordinate forms a Sidon violation in $R$. But $R$ is also Sidon.

    If each side has one variable in $P$ and one in $Q$, then looking at the second coordinate of the variables in $P$, they must be equal. Thus the $Q$ variables are equal as well, and the Sidon property holds.

   Consider now any equation $F: \sum_{i=1}^L c_i x_i = 0 \in \cL$. Consider any nontrivial solution $x_1,\dots,x_L \in P \cup Q$ to $F$, and let $I \subset [L]$ be the indices $i$ for which $x_i \in P$. We will consider two cases and derive a contradiction in each, showing that no such solution can exist.

   \textbf{Case $I = \emptyset$.} Then this solution lies entirely in $Q$, and is forbidden by considering the first coordinate and properties of $R$.

   \textbf{Case $I \neq \emptyset$.} Recalling that $\sum_i \abs{c_i} < M$, so there is no torsion in the $\ZZ_M$ coordinate, it follows that
   $$
   \sum_{i \in I} c_i = 0,
   $$
   and thus by assumption that $F \in \cL$ that $\abs{I} = 2$. Without loss of generality assume $I = \set{1,2}$ so that $c_2 = -c_1$. If $L=2$ then clearly $x_1 = x_2$ and any solution is trivial, so assume $L > 2$.
   
   We first argue that $x_1 = x_2$ always. Since the second coordinates of $x_3,\dots,x_L \in Q$ are all zero, looking at the second coordinate of
   $$
   \sum_{i=1}^L c_i x_i = 0,
   $$
   it follows that the second coordinate of $x_1$ equals the second coordinate of $x_2$. Then $x_1 = x_2$ by definition of $P$.
   
   Since $c_2 = -c_1$, it then follows that
   $$
   \sum_{i =3}^L c_i y_i = 0,
   $$
   where $y_i \in R$. Note that $\sum_{i =3}^L c_i \neq 0$, as otherwise all of $[L]$ would be a zero-sum subset of more than two coefficients. Thus any solution to this subequation in $R$ is nontrivial, contradicting \Cref{lem:randomsidon} as the subequation lies in $\cL$.

   \end{proof}

\subsection*{Acknowledgements}

I would like to thank Jacob Fox for pointing out the proof of the lower bound on $\chi(\ZZ_N,D)$ in \Cref{thm:exeq}, which I had originally proved with spectral methods, and Sarah Peluse for helpful discussions.

\bibliographystyle{amsplain}
\bibliography{Biblio.bib}

\end{document}